\documentclass[12pt]{article}
\usepackage{amssymb,amsfonts,amsmath,latexsym,epsf,url}
\usepackage[usenames,dvipsnames]{pstricks}
\usepackage{pstricks-add}
\usepackage{subfigure}
\usepackage{graphicx}
\usepackage{color,soul}
\usepackage{epsfig}
\usepackage{tikz}
\usepackage[colorlinks]{hyperref}
\usepackage{cite}
\usepackage{algorithm}
\usepackage{algpseudocode}
\usepackage{graphicx}
\usepackage{booktabs}
\usepackage{multirow}

\newtheorem{theorem}{Theorem}[section]

\newtheorem{proposition}[theorem]{Proposition}

\newtheorem{definition}[theorem]{Definition}

\newcommand{\qed}{\hfill $\square$\medskip}

\begin{document}

\def\nt{\noindent}

\title{Some results on Hamming graphs and an extended Hamming graphs}

\author{
	Ali Zafari$^1$\footnote{Corresponding author} \and	
Saeid Alikhani$^{2}$
}


\maketitle

\begin{center}

$^1$Department of Mathematics, Faculty of Science,
Payame Noor University, P.O. Box 19395-4697, Tehran, Iran\\ 
{\tt zafari.math@pnu.ac.ir}
\medskip

$^{2}$Department of Mathematical Sciences, Yazd University, 89195-741, Yazd, Iran\\
{\tt alikhani@yazd.ac.ir}

\end{center}
	
\begin{abstract}
	In this paper we first obtain the spectrum of the folded hypercube in a new approach. Then we introduce a new family of graphs called the extended Hamming graph, denoted by $EH(n,2^n)$, which is constructed from the well-known Hamming graph $H(n,2^n)$.
The graph $EH(n,2^n)$ shares the same vertex set as $H(n,2^n)$ but includes additional edges, called complementary edges, connecting each $n$-tuple vertex $u$ to its complement $u^c$, where $u^c$ is defined such that the sum of each two corresponding coordinates of $u$ and $u^c$ equals $2^n-1$. 
We investigate several algebraic and structural properties of this new family of graphs. Specifically, we show that the diameter of $EH(n,2^n)$ is $n$.
We prove that $EH(n,2^n)$ is a Cayley graph, but we demonstrate that it is not a distance regular graph. Finally, we determine the spectrum of 
$EH(n,2^n)$, showing that its eigenvalues are $\lambda_i\pm 1$, where $\lambda_i$ are the eigenvalues of the underlying Hamming graph $H(n,2^n)$. 
The multiplicity of each eigenvalue is explicitly calculated.
\end{abstract}

\noindent{\bf Keywords:} Hamming graph, Extended Hamming graph, Cayley graph.

\medskip
\noindent{\bf AMS Subj.\ Class.:}  05C25, 94C15.


\section{Introduction}
\label{sec:introduction}

Graph theory provides a powerful framework for modeling relationships between discrete objects, and certain families of graphs, due to their rich structure and symmetry, have been central to areas like computer science, coding theory, and network design. Among these, the Hamming graphs represent a fundamental class.

The Hamming graph $H(n,m)$ is defined on the vertex set $\Omega^n$, where $\Omega=\{0, 1, ..., m-1\}$.
Two vertices are adjacent if they differ in exactly one coordinate. It is well known that $H(n,m)$ is the Cartesian product of $n$ complete graphs 
$K_m$, i.e., $K_m\square \cdots \square K_m$. A notable special case is $H(n,2)$, which is the $n$-dimensional hypercube $Q_n$.
Variants of Hamming graphs, such as the folded hypercube $FQ_n$ was proposed first in \cite{Amawy-1}, have been introduced to enhance topological properties like diameter,
which can be halved in the folded hypercube compared to the standard hypercube. 

Hamming graphs are an important class of Cayley graphs \cite{Hahn-1,Liu-1}. These graphs have been and still are studied intensively because of their symmetry properties and their connections to communication networks, quantum physics, coding theory, and other areas (see, e.g., \cite{Hahn-1,Li-Liu-1}). However, only few results on generalized Hamming graphs have been obtained so far. In 2010, Sander studied the eigenspaces of  Hamming graphs and unitary Cayley graphs in some cases, see \cite{Sander-1}. In the same year, Klotz and Sander observed that $H(m_1, \dots, m_r; D)$ is an integral Cayley graph \cite{Klotz-1}. In 2023, Li, et al. provided  the square of generalized Hamming graphs by the properties of abelian groups \cite{Li-Liu-2}.
  
This paper focuses on the Hamming graph $H(n,2^n)$ and introduces a novel structure derived from it: the extended Hamming graph $EH(n,2^n)$.
The motivation for introducing $EH(n,2^n)$ is to explore a new highly symmetric family of graphs and to analyze how the addition of complementary edges impacts its algebraic and combinatorial properties. First, we state the following definition of the extended Hamming graph.

\begin{definition}
We denote the extended Hamming graph by $EH(n,2^n)$. The vertex set of $EH(n,2^n)$ is the same as that of $H(n,2^n)$. A complementary edge connects an $n$-tuple vertex $u=(u_1,\ldots,u_n)$ to its complement
$u^c=(2^n-1-u_1, \ldots, 2^n-1-u_n)$, provided they are distinct. So, the edge set of 
$EH(n,2^n)$ is $E(H(n,2^n))\cup \{u,u^c\}$. 
\end{definition} 
This construction is a generalization of the folded graph principles to the context of $H(n,2^n)$.
Since the valency of each vertex $v$ of the Hamming graph $H(n,m)$ is $n(m-1)$,
so the valency of each vertex $v$ of the Hamming graph $H(n,2^n)$ is $n(2^n-1)$. In particular,  the valency of each vertex $v$ of the extended Hamming graph $EH(n,2^n)$ is $n(2^n-1)+1$. We can verify for $n>1$ the Hamming graph  $H(n,2^n)$ is not bipartite.

The paper is structured as follows: 

Section 2 provides  preliminaries from graph theory and algebraic graph theory that are essential for understanding the main results. In Section 3, we obtain the spectrum of the folded hypercube.    Section 4 contains the main results about the extended Hamming graphs, including the proof of the diameter, the Cayley graph property, the non-distance-regularity, and the full spectrum calculation of $EH(n,2^n)$. Finally, we conclude in Section 5.  

\section{Preliminaries and Definitions}
This section introduces the fundamental concepts, notation, and prior results from graph theory and algebraic graph theory that are essential
for understanding the main results presented in this paper. In this paper, a graph $\Gamma=(V(\Gamma),E(\Gamma))$ is considered 
as an undirected simple graph with the vertex-set $V(\Gamma)$ and  the edge-set $E(\Gamma)$. For all the terminology
and notation not defined here, we follow \cite{{Biggs-1,Godsil-1}}.

The {diameter} of a connected graph is the maximum distance between any two vertices. 
The adjacency matrix of graph $\Gamma$ of order $n$, denoted by $A(\Gamma)$, is a $0$-$1$ matrix of order $n$ with entries $a_{ij}$ such that
$a_{ij} = 1$ if the $i^{th}$ and $j^{th}$ vertices are adjacent and $a_{ij} = 0$ otherwise. The eigenvalues of $A(\Gamma)$ are called the eigenvalues 
of $\Gamma$, and the collection of eigenvalues of $\Gamma$ with multiplicities is called the spectrum of $\Gamma$. A graph is integral \cite{Abdollahi-1,Ahmadi-1,Balinska-1,Harary-1,Wang-1} if all its eigenvalues are integers.

Let recall some definitions. 
Let $G$ be a finite group with identity $1$ and let $\Omega$ be a inverse-closed subset of $G$ with the properties:
\begin{itemize}
    \item $x \in \Omega \implies x^{-1} \in \Omega$,
    \item $1 \notin \Omega$.
\end{itemize}
The Cayley graph $\Gamma = \operatorname{Cay}(G, \Omega)$ is a  graph whose vertex-set and edge-set are defined as follows:
$$
V(\Gamma) = G, \qquad E(\Gamma) = \big\{ \{g, h\} \mid g^{-1}h \in \Omega \big\}.
$$


A regular graph $\Gamma$ with diameter $d$ is {distance regular} if, for any two vertices $u$ and $v$ at distance $r$, the following intersection numbers are constant:
\begin{itemize}
		\item $c_r = |\Gamma_{r-1}(v) \cap \Gamma_1(u)|,$
		\item $b_r = |\Gamma_{r+1}(v) \cap \Gamma_1(u)|.$
\end{itemize}
The partition of vertices based on distance from a fixed vertex $v$, denoted $\Gamma_r(v) = \{\Gamma_0(v), \ldots, \Gamma_d(v)\}$, is called the {distance partition}.

Let $n \ge 3$ be an integer. The hypercube $Q_n$ of dimension $n$ is the graph with the vertex set $\{(x_1, x_2, \dots, x_n) \mid x_i \in \{0,1\}\}$, where two vertices $(x_1, x_2, \dots, x_n)$ and $(y_1, y_2, \dots, y_n)$ are adjacent if and only if $x_i = y_i$ for all but one $i$.
The folded hypercube $FQ_n$ of dimension $n$ is the graph obtained from the hypercube $Q_n$ by adding edges, called complementary edges, between any two vertices $x = (x_1, x_2, \dots, x_n)$ and $y = (\bar{x}_1, \bar{x}_2, \dots, \bar{x}_n)$, where $\bar{1} = 0$ and $\bar{0} = 1$. 

Let $G$ and $H$ be two graphs with the adjacency matrices $A(G)$ and $A(H)$.
The Cartesian product of two graphs $G$ and $H$ has adjacency matrix:
$$
A(G \square H) = A(G) \otimes I_{n_2} + I_{n_1} \otimes A(H)
$$
where $n_1$ and $n_2$ are the numbers of vertices in $G$ and $H$, respectively, and $\otimes$ denotes the Kronecker product. This corresponds to the
tensor product in the space of linear operators.


\section{The spectrum of folded hypercube}

The spectrum of folded hypercube studied in \cite{Chen-1,Mirafzal-1}. Here in Theorem \ref{b.2}, we give the spectrum of $FQ_n$ with another approach. 
Let recall the spectrum of $Q_n$. 
\begin{theorem}\label{b.1}{\rm\cite{{Brower-1}}}
	The spectrum of hypercube  $Q_n$,  is
	$$
	\{\lambda_{i}^{m({\lambda_i})}|\ \lambda_{i}=n-2i,\ m({\lambda_i})=\binom{n}{i},\ 0\leq i\leq n\},
	$$
	where $n$ is the diameter of hypercube  $Q_n$. 
\end{theorem}

\begin{theorem}\label{b.2}
	The spectrum of folded hypercube  $FQ_n$, is
	
	$$
	\{\theta_{i}^{ m({\theta_i})}|\ \theta_{i}=\lambda_{i}+(-1)^i,\ m({\theta_i})=\binom{n}{i},\ 0\leq i\leq n\},
	$$
	where $\lambda_{i}$ is an eigenvalue of the hypercube  $Q_n$.
\end{theorem}
\begin{proof}
	Let $\Gamma=Q_n$ and $\Lambda=FQ_n$.  First, we show that all the eigenvalues of  the folded  $FQ_n$  are 
	$$
	\theta_{i}=\lambda_{i}+(-1)^i \quad \text{for } 0\leq i\leq n,
	$$
	and then we show that the multiplicity $m(\theta_i)=\binom{n}{i}$, for $0\leq  i\leq n$. 
	For this purpose, let $A_\Gamma$ and $A_\Lambda$ be the adjacency matrices of the hypercube  $Q_n$ and folded hypercube  $FQ_n$, respectively. Hence, 
	$A_\Lambda=A_\Gamma+ D$, where $D$ is equal to the  $(q)\times(q)$ matrix
	\begin{center}
		$D=\begin{bmatrix}
		0 & 0 &0 &0 &...&0&0&0&1 \\
		0 & 0 & 0&0&...&0&0&1&0 \\
		0& 0 & 0&0& &0&1&0&0 \\
		&  & &...& \\
		&  & &  & \\
		0& 0 &1&0 &...&0 &0&0&0 \\
		0&1&0&0  & ...&0&0&0&0 \\
		1&0&0& 0 & ...& 0&0&0&0 \\
		\end{bmatrix},$
	\end{center}
	where $q=2^n$. It is well known that, hypercube  $Q_n$ is a distance-transitive graph, and hence hypercube $Q_n$ is a distance regular graph
	with the diameter $n$, see  ~\cite{Brower-1}. Hence hypercube $Q_n$ has $n+1$ eigenspace $V_0, ..., V_n$ so that 
	$dim(V_i)=\binom{n}{i}$, for $0\leq  i\leq n$. Based on \cite{Bannai-1,Delsarte-1}, hypercube  $Q_n$ belongs to a symmetric association
	scheme generated by distance adjacency matrices $A_0, A_1, \dots, A_n$, then $A_n=D$. Therefore, $A_\Gamma$ and $D$ are commute and  diagonalizable, hence they are simultaneously diagonalizable, and hence each eigenspace $V_i$ is invariant on $D$. Also, from Page 209 of \cite{Bannai-1}, 
	the eigenvalues of $A_n=D$ on $V_i$ in hypercube  $Q_n$ are given by krawtchouk polynomial
	$$
	K_n(i) = \sum_{j=0}^{n} (-1)^j (q-1)^{n-j}\binom{i}{j} \binom{n-i}{n-j}.
	$$
	We can verify that, $\binom{i}{j}=0$ if $j>i$ and  $\binom{n-i}{n-j}=0$ if $i>j$. Hence, if $i=j$, 
	then we have $K_n(i)=(-1)^i$. So $D$ acts as a scalar on each $V_i$, specifically $D|_{V_i} = (-1)^i \cdot I$. 
	Thus the eigenvalues of $D$ on the $i^{th}$ eigenspace of  $A_\Gamma$ are $(-1)^i$.
	Now, let $Z=\{z_1, z_2, ... , z_{q}\}$  such that each $z_i\in Z$ be an eigenvector of $A_\Gamma$ and $D$. 
	Hence, for any eigenvector $z_i\in Z$ in the $i^{th}$ eigenspace of $A_\Gamma$ we have $A_\Gamma z_i=\lambda_iz_i$ and $Dz_i=(-1)^iz_i$, 
	therefore 
	$$
	A_\Lambda z_i=(\lambda_i+(-1)^i)z_i.
	$$ 
	Since $D$ acts as a scalar on this eigenspace, then the multiplicity $m(\theta_i)=\binom{n}{i}$, for $0\leq  i\leq n$. \qed
\end{proof}

\section{Extended Hamming graphs}

In this section, we obtain the diameter, the Cayley graph property, the non-distance-regularity, and the full spectrum calculation of $EH(n,2^n)$. We start with diameter of $EH(n,2^n)$.

\begin{theorem}\label{c.1}
The diameter of the extended Hamming graph $EH(n,2^n)$ is $n$.
\end{theorem}
\begin{proof}
Let $\Omega=\{0, 1, ..., 2^{n}-1\}$. We can verify that the diameter of the extended Hamming  graph $EH(n,2^n)$ is at most $n$, because the 
Hamming  graph $H(n,2^n)$ is a subgraph of extended Hamming  graph $EH(n,2^n)$, and the diameter of the Hamming  graph $H(n,2^n)$ is equal to $n$. 
Now, if we consider two distinct $n$-tuple vertices
$u=(0, ..., 0)$ and  $v=(1, ..., 1)$ in Hamming  graph $H(n,2^n)$, then the distance between two distinct $n$-tuple vertices $u=(0, ..., 0)$ and  
$v=(1, ..., 1)$ is $n$, that is $d_H(u,v)=n$. Since in the extended Hamming  graph $EH(n,2^n)$, the $n$-tuple $u=(0, ..., 0)$ is adjacent
to $n$-tuple $u^c=(2^n-1, ..., 2^n-1)$, then there is a path $uu^c ...v$ so that 
$$
d_H(u,v)=d_H(u,u^c)+d_H(u^c,v)=1+d_H(u^c,v).
$$ 
Hence,  
$$
d_{EH}(u,v)= min\{d_H(u,v),1+d_H(u^c,v)\}=min\{n,1+n\}.
$$ 
Therefore, the diameter of the extended Hamming  graph $EH(n,2^n)$ is $n$.\qed
\end{proof}
\begin{theorem}\label{c.2}
The extended Hamming graph $EH(n,2^n)$ is a Cayley graph. 
\end{theorem}
\begin{proof}
Let $\Gamma=EH(n,2^n)$, and let $m=2^n$. There is a unique field of order $2^n$ we denoted by 
$\mathbb{F}_{2^n}$ so that $(\mathbb{F}_{2^n}, +)$ is abelian group of 
order $2^n$ and $(\mathbb{F}_{2^n})^*=\mathbb{F}_{2^n}-\{0\}$ is a cyclic group of order $2^n-1$. 
Also, let $G$ be the additive  group $(\mathbb{F}_{2^n})^n$, that is $G=(\mathbb{F}_{2^n})^n$, where $(\mathbb{F}_{2^n})^n$ which means
the direct product of $n$ copies of the additive group $(\mathbb{F}_{2^n},+)$. Hence each element of $G$ has order $2$,
in other words $G$ is an elementary abelian $2$-group. Now, let $e_i$ be a $n$-tuple so that $i^{th}$ coordinate of $e_i$ is 
equal to $1$ and other coordinates of $e_i$ is equal to $0$, that is  
$$
e_i=(0, ...0, 1,0, ... 0),
$$ 
which means $e_i=(0, ...0, 1,0, ... 0)$ is the multiplicative identity in $i^{th}$ copy of the direct product of $G$. 
So, if we consider the set $S$ as follows:
$$
S = \{ g e_i=g(0, ...0, 1,0, ... 0) \mid 1 \leq i \leq n,\ g \in \mathbb{F}_{2^n}-\{0\} \} \cup \{(1, ... ,1) \},
$$
where $(1, ... ,1)\in(\mathbb{F}_{2^n})^n$, then we can verify that the set $S$ is inverse closed subset of $G-(0, ..., 0)$, 
because $G$ is an elementary abelian $2$-group and hence for each $g\in G$ we have $2g=(0, ..., 0)$.
Also, we can verify that the cardinality of the set $S$ is 
equal to the valency of the extended Hamming  graph $EH(n,2^n)$. Now, we show that the extended Hamming  graph $EH(n,2^n)$  
is isomorphic to Cayley graph $Cay(G, S)$ is defined already. For this purpose, let $\Omega=\{0, 1, ..., 2^{n}-1\}$, and
let $\varphi: \Omega \to \mathbb{F}_{2^n}$ be a bijection satisfying
$$
\varphi(2^n - 1 - x) = \varphi(x) + 1 \quad \text{for all } x \in \Omega.
$$
Now, if we define the map $\Phi: \Omega^n \to G$ by rule
$$
\Phi(x_1, \dots, x_n) = (\varphi(x_1), \dots, \varphi(x_n)),
$$
then we can verify that $\Phi$ is an isomorphism between the extended Hamming graph $EH(n,2^n)$ and the Cayley graph $Cay(G, S)$.
\qed
\end{proof}
\begin{proposition} \label{c.3}
The extended Hamming graph $EH(n,2^n)$ cannot be a distance regular graph.   
\end{proposition} 
\begin{proof}
Let  $\Gamma=EH(n,2^n)$, and let 
$$
\Gamma_r(u)=\{\Gamma_0(u), ... , \Gamma_{n}(u)\},
$$
be the distance partition of $\Gamma$ for each $n$-tuple $u$ of $\Gamma$, where $r$ is a non-negative integer not exceeding $n$, 
the diameter of $\Gamma$. Without loss of generality if we consider the $n$-tuple vertex $u=(0, ...,0)$ in the extended Hamming graph
$\Gamma$, then its neighbors are the $n$-tuple vertices so that differ in exactly one coordinate and its complement say 
$u^c=(2^n-1, ...,2^n-1)$. 
So, if we take $v=(1, ...,0)$, then we have $d_{EH}(u,v)=d_{EH}(u,u^c)=1$, and hence we can verify that 
$$
|\Gamma_{2}(u)\cap \Gamma_{1}(v)|\neq|\Gamma_{2}(u)\cap \Gamma_{1}(u^c)|.
$$
Thus, the extended Hamming  graph $EH(n,2^n)$ cannot be a distance regular graph. \qed
\end{proof}
\begin{theorem}\label{c.4}{\rm(\cite{{Brower-1}}, Page 261)}
The spectrum of the  Hamming graph $H(n,2^n)$,  is
$$
\{\lambda_{i}^{m(\lambda_{i})}|\ \lambda_{i}=2^n(n-i)-n,\ m({\lambda_i})=\binom{n}{i}(2^n-1)^i,\ 0\leq i\leq n\},
$$
\end{theorem}
\begin{theorem}\label{c.5}
The spectrum of the extended Hamming graph $EH(n,2^n)$,  is
$$
\{\theta_{i,t}^{m({\theta_{i,t}})}|\ \theta_{i,t}=\lambda_{i}+(-1)^t,\ m({\theta_{i,t}})=\binom{n}{i}\binom{i}{t}2^{(n-1)t}(2^{(n-1)}-1)^{(i-t)},\ 0\leq i\leq n\,\  0\leq t\leq i\},
$$
where $\lambda_{i}$ is an eigenvalue of Hamming graph $H(n,2^n)$.
\end{theorem}
\begin{proof}
Let $\Lambda=H(n,2^n)=K_{2^n} \square \cdots \square K_{2^n},$
and let $\Gamma=EH(n,2^n)$. By the following cases first we show that $\Gamma$ is an integral graph
and then we compute the multiplicity of each eigenvalue of $\Gamma$.

Case 1.
We show that the extended Hamming graph $EH(n,2^n)$  is an integral graph.

Let $A_\Lambda$, $A_\Gamma$ and $A_C$ be the adjacency matrices of graphs  $\Lambda$, $\Gamma$ and the complementary edges of graph $\Gamma=EH(n,2^n)$, respectively. 
Hence, $A_\Gamma=A_\Lambda+ A_C$, where $A_C$ is equal to the  $2^{(n^2)}\times2^{(n^2)}$ matrix
\begin{center}
			$A_C=\begin{bmatrix}
			0 & 0 &0 &0 &...&0&0&0&1 \\
			0 & 0 & 0&0&...&0&0&1&0 \\
			0& 0 & 0&0& &0&1&0&0 \\
			&  & &...& \\
			&  & &  & \\
			0& 0 &1&0 &...&0 &0&0&0 \\
			0&1&0&0  & ...&0&0&0&0 \\
			1&0&0& 0 & ...& 0&0&0&0 \\
\end{bmatrix}.$
\end{center}
We can verify that $A_\Lambda$ and $A_C$ are commute and diagonalizable, hence they are simultaneously diagonalizable. 
So, if $V_i$ is an eigenspace of $A_\Lambda$ for an eigenvalue $\lambda_i(\Lambda)$ with dimension $dim(V_i)$, 
then $A_C$ preserves each $V_i$. Since, $(A_C)^2=I$, then the eigenvalues of $A_C|_{V_i}=\pm1$. So if we consider 
$$
V=V^+\oplus V^-,
$$
where $V^{\pm}$ are $\pm1$ eigenspaces of $A_C$, then all the eigenvalues of 
the adjacency matrix $A_\Gamma$ of the extended Hamming graph $EH(n,2^n)$ are $\lambda_i\pm 1$. \\
 
Case 2. 
Suppose the multiplicity of $\lambda_i+ 1$ is equal to $m^{+}_i$, and the multiplicity of $\lambda_i- 1$ is equal to $m^{-}_i$. 
We show that, the spectrum of the extended Hamming graph $EH(n,2^n)$,  is
$$
\{\theta_{i,t}^{m({\theta_{i,t}})}|\ \theta_{i,t}=\lambda_{i}+(-1)^t,\ m({\theta_{i,t}})=\binom{n}{i}\binom{i}{t}2^{(n-1)t}(2^{(n-1)}-1)^{(i-t)},\ 0\leq i\leq n\,\  0\leq t\leq i\}.
$$
For this purpose, we can verify that the adjacency matrix of graph 
$$
\Lambda=K_{2^n} \square \cdots \square K_{2^n},
$$
can be written as 
$$
A_\Lambda = \sum_{k=1}^{n} I_{2^n}^{\otimes(k-1)} \otimes (J_{2^n} - I_{2^n}) \otimes I_{2^n}^{\otimes(n-k)},
$$
where $J_{2^n}$ is the all-ones matrix of size $2^n \times 2^n$, $I_{2^n}$ is the identity matrix, and $I_{2^n}^{\otimes k}$ denotes
the $k$-fold Kronecker product of $I_{2^n}$ with itself. Now, let 
$$
D=J_{2^n} - I_{2^n}
$$
be the adjacency matrix of complete graph $K_{2^n}$, and let $P_{2^n}$ be the permutation matrix representing
$$
x\mapsto 2^n-1-x,\qquad x\in\{0,1,\dots,2^n-1\}.
$$
Hence $P^2_{2^n}=I$, and hence all the eigenvalues of $P_{2^n}$ are $\mu_{p}=\pm1$ with the same multiplicity as $2^{n-1}$.
In particular, we can verify that $A_C=P_{2^n}^{\otimes n}$. Now we show that $A_C$ acts on $\prod_{p=1}^{n} \mu_p$.
Since, the adjacency matrix $D$ of complete graph $K_{2^n}$ has eigenvalues $2^n-1$ and $-1$ with the multiplicity 
$1$ and  $2^n-1$, respectively, then the eigenvector $1_{2^n}$ (all-ones) of the adjacency matrix $D$ is correspond to the eigenvalue $2^n-1$. So, if
$W_R\subseteq \mathbb{R}^{2^n}$ is a span of $1_{2^n}$ correspond to the eigenvalue $2^n-1$ of complete graph $K_{2^n}$, and 
$W_L\subseteq \mathbb{R}^{2^n}$ is a orthogonal complement correspond to the eigenvalue $-1$ of complete graph $K_{2^n}$, 
that is for each $w_l\in W_L$ we have $w_l 1_{2^n}=0$, then eigenvectors of  $A_\Lambda$ are
$$
w_1\otimes ...\otimes w_n,
$$
where $w_k\in W_R$ or $w_k\in W_L$. On the other hand, if
$u=(u_1, ..., u_n)$ is a vertex of  graph $H(n,2^n)$, then 
$$
\sigma=\sigma_1\otimes ...\otimes \sigma_n,
$$
by rule 
$$
\sigma(u)=(\sigma_1(u_1), ...,\sigma_n(u_n))=(2^n - 1 - u_1,...,2^n - 1 - u_n ),
$$ 
is an automorphism of graph $H(n,2^n)$, where $\sigma_r$ is a permutation matrix of $r$-factor  $ K_{2^n}$ of $\Lambda$. 
Since, the permutation matrix $P_{2^n}$ and the adjacency matrix $D$ of complete graph $K_{2^n}$ are commute and diagonalizable, 
hence they are simultaneously diagonalizable. Therefore, $P_{2^n}$ preserves the eigenspaces of the adjacency matrix $D$, 
and hence the  permutation $\sigma_r$ preserves $W_R$ and $W_L$. Thus,
$$
\sigma_r|{W_R}=1 \text{ and } \sigma_r|{W_L}=\pm 1,
$$
that is the eigenvector $1_{2^n}$ in $W_R$ correspond to the eigenvalue $1$  of $P_{2^n}$, and each eigenvector of $W_L$ correspond
to the eigenvalues $1$ and $-1$ of $P_{2^n}$, with the multiplicity $2^{n-1}-1$ and $2^{n-1}$, respectively.
Now, let $W_L$ decomposes as 
$$
W_L = W_L^{+}\oplus W_L^{-},
$$ 
where
$$
\dim W_L^{+} = 2^{\,n-1}-1,\qquad \dim W_L^{-} = 2^{\,n-1},
$$
corresponding to the positive eigenvalues $1$  of $P_{2^n}$ with the multiplicity 
$2^{n-1}-1$ and the negative eigenvalue $-1$ of $P_{2^n}$ with the multiplicity 
$2^{n-1}$ for $\sigma_r$, respectively. Thus, $A_C$ acts on $\prod_{p=1}^{n} \mu_p$. 
Moreover in 
$$
\sigma=\sigma_1\otimes \cdots \otimes \sigma_n,
$$ 
if $i$ of $\sigma_r$ are in $W_L$ so that $t$ number of them are in $W_L^{-}$, that is eigenvalue $-1$ for $\sigma_r$, 
and $i-t$ of them  are in  $W_L^{+}$, 
that is eigenvalue $1$ for $\sigma_r$, then  $n-i$ of $\sigma_r$ are in $W_R$. So, eigenvalue of $A_C$ on tensor product is
$$
(+1)^{n-i}\cdot(+1)^{i-t}\cdot(-1)^{t}=(-1)^t,
$$
and so  for fixed $i$ and $t$, if
$$
\theta_{i,t}=\lambda_{i}+(-1)^t,
$$ 
is an eigenvalue of the extended Hamming graph $EH(n,2^n)$ then  we need to consider each choice of eigenvalue for each factor. 
There are $i$ positions in $W_L$ with $\binom{n}{i}$ choices, among these $i$ positions, there are $t$ positions in $W_L^{-}$, 
with $\binom{i}{t}$ choices and a basis vector in each chosen space: 
$$\begin{cases}
\dim(W_L^{-})=2^{\,n-1} &\text{ for each of the $t$ positions},\\[2pt]
\dim(W_L^{+})=2^{\,n-1}-1 &\text{ for each of the $i-t$ positions}.
\end{cases}
$$
Hence by multiplication principle we have
$$
m(\theta_{i,t})=\binom{n}{i}\binom{i}{t}\; 2^{(n-1)t}\,(2^{\,n-1}-1)^{\,i-t}.
$$ \qed
\end{proof}

\section{Conclusion}

In this paper, we introduced and explored the properties of the extended Hamming graph, denoted as $EH(n,2^{n})$, a novel graph structure derived from the well-known Hamming graph $H(n,2^{n})$. Our investigation was motivated by the desire to analyze how the addition of complementary edges impacts the algebraic and combinatorial characteristics of this important family of graphs.

We first revisited the spectrum of the folded hypercube $FQ_n$ using a new approach, confirming that its eigenvalues $\theta_i$ are related to the hypercube eigenvalues $\lambda_i$ by $\theta_i = \lambda_i + (-1)^i$ with multiplicity $m(\theta_i)=\binom{n}{i}$. 

The main focus, however, was on the extended Hamming graph $EH(n,2^{n})$. We established several key algebraic and structural properties of this new family of graphs:

\begin{enumerate}

\item We proved that the diameter of $EH(n,2^{n})$ is $n$.
\item We demonstrated that $EH(n,2^{n})$ is a Cayley graph, leveraging the direct product of the additive group $(\mathbb{F}_{2^{n}},+)$.
\item We showed that $EH(n,2^{n})$ is not a distance regular graph.
\item Finally, we determined the complete spectrum of $EH(n,2^{n})$, showing that its eigenvalues $\theta_{i,t}$ are $\lambda_i + (-1)^t$, where $\lambda_i$ are the eigenvalues of $H(n,2^{n})$. We explicitly calculated the multiplicity of each eigenvalue.
\end{enumerate} 
These results provide a foundation for further analysis of this new class of highly symmetric graphs. Future work could include exploring other properties, such as the  number of independent sets, dominating sets or the connectivity, and investigating relationship between parameters of $EH(n,2^n)$.

\section*{Funding and Conflict of interest} 
The authors have not disclosed any funding and declare no conflict of interest.


\bigskip
\end{document}